\documentclass{amsart}

\usepackage{verbatim}
\usepackage{amsmath, amsfonts, amsthm}
\usepackage{color}
\usepackage{graphicx}

\DeclareMathOperator{\MF}{MF}
\DeclareMathOperator{\ind}{ind}

\newtheorem*{thm}{Theorem}
\newtheorem{lem}{Lemma}

\begin{document}
\bibliographystyle{mrl}

\title[The asymptotic density of Wecken maps]{The asymptotic density of Wecken maps on surfaces with boundary}
\author{Seung Won Kim}
\address{Dept. of Mathematics and Applied Statistics, Kyungsung University, Busan 608-736, Republic of Korea}
\email{kimsw@ks.ac.kr} 
\author{P. Christopher Staecker}
\address{Dept. of Mathematics, Fairfield University, Fairfield CT, 06824, USA}
\email{cstaecker@fairfield.edu}
\urladdr{http://faculty.fairfield.edu/cstaecker/}%%
\keywords{fixed points, surface, asymptotic density, Nielsen theory}
\thanks
{The first author was supported by Basic Science Research Program through the
National Research Foundation of Korea funded by Ministry of Education(NRF-2014R1A1A2058873).}
\subjclass[2010]{55M20, 37C20}

\begin{abstract}
The Nielsen number $N(f)$ is a lower bound for the minimal number of fixed points among maps homotopic to $f$. When these numbers are equal, the map is called Wecken. The paper~\cite{bgms12} by Brimley, Griisser, Miller, and the second author investigates the abundance of Wecken maps on surfaces with boundary, and shows that the set of Wecken maps has nonzero asymptotic density.

We extend the previous results as follows: When the fundamental group is free with rank $n$, we give a lower bound on the density of the Wecken maps which depends on $n$. This lower bound improves on the bounds given in the previous paper, and approaches 1 as $n$ increases. Thus the proportion of Wecken maps approaches 1 for large $n$. In this sense (for large $n$) the known examples of non-Wecken maps represent exceptional, rather than typical, behavior for maps on surfaces with boundary.
\end{abstract}

\maketitle

Given a complex $X$ and a selfmap $f:X\to X$, the Nielsen number $N(f)$ is a  homotopy invariant which satisfies $N(f) \le \MF(f)$, where $\MF(f)$ is the minimal number of fixed points for any mapping in the homotopy class of $f$. Nielsen defined his number in the 1920s as more computable alternative to $\MF(f)$. In the 1940s Wecken \cite{weck41} proved that in fact $N(f) = \MF(f)$ for compact manifolds of dimension not equal to 2. It was a long standing question whether Wecken's theorem holds in dimension 2, and this was finally answered in the 1980s by Jiang \cite{jian84} who gave an example of a map on the pants surface with $N(f)=0$ but $\MF(f)=2$.

Maps on surfaces which satisfy $N(f)=\MF(f)$ are called \emph{Wecken maps}, and this paper continues the work of Brimley, Griisser, Miller, and the second author in \cite{bgms12} investigating the abundance (or lack thereof) of the Wecken maps among the set of all selfmaps on surfaces with boundary.

For maps on compact surfaces with boundary, the numbers $N(f)$ and $\MF(f)$ depend only on the induced homomorphism on the fundamental group, which is a finitely generated free group. When $\phi$ is the induced homomorphism of a Wecken selfmap, we say that $\phi$ is \emph{Wecken}. For each $n\ge 2$, let $W_n$ be the set of all Wecken endomorphisms on the free group of rank $n$. (For $n=1$ the question is uninteresting because all maps on the circle $S^1$ are known to be Wecken.)

We will measure the size of $W_n$ according to its asymptotic density within the set of all endomorphisms. We will review the basic definitions of asymptotic density:
For a finitely generated free group $G$, let $G_p$ be the subset
of all words of $G$ of length at most $p$. The \emph{strict asymptotic density} of a subset $S\subset G$ is defined as 
\[ \rho(S) = \lim_{p \to \infty} \frac{|S \cap G_p|}{|G_p|}, \]
where $|\cdot|$ denotes the cardinality. The set $S$ is said to be
\emph{generic} if $\rho(S) = 1$. The above limit may not exist, in which case we can discuss the related density, which always exists: 
\[ D(S) = \liminf_{p\to\infty} \frac{|S \cap G_p|}{|G_p|}. \]
Of course when $\rho(S)$ does exist we have $D(S)=\rho(S)$.

The density can be thought of as the probability that a random element of $G$ is in the set $S$.

Similarly, if $S \subset G^k$ is a set of $k$-tuples of
elements of $G$, the strict density of $S$ is defined as
\[ \rho(S) = \lim_{p \to \infty} \frac{|S \cap (G_p)^k|}{|(G_p)^k|}, \]
and $S$ is called \emph{generic} if $\rho(S) = 1$. Similarly we define $D(S)$ using $\liminf$ in place of $\lim$.

An endomorphism $\phi:G \to G$ is equivalent combinatorially to a $n$-tuple of elements of $G$
(the $n$ elements are the words $\phi(a_1), \dots, \phi(a_n)$). Thus
the asymptotic density of a set of 
homomorphisms can be defined in the same sense as above, viewing the set of
homomorphisms as a subset of the product set $G^n$. 

Our main result is the following theorem:
\begin{thm}
For $n\ge 2$,
\[ \lim_{n\to\infty} D(W_n) = 1. \]
\end{thm}

Our general approach is based on Wagner's algorithm of \cite{wagn99} for computing the Nielsen number, which we will briefly review. Given an endomorphism $\phi:G \to G$, we first build the set of \emph{Wagner tails}, which are elements of $G$ arising in pairs. This set is constructed as follows: for any $\phi$, the trivial elements $w_0 = 1, \bar w_0 = 1$ are Wagner tails. Also, for each occurrence of the letter $a_i^\epsilon$ (for $\epsilon \in \{+1, -1\}$) in $\phi(a_i)$, write a reduced product $\phi(a_i) = va_i^{\epsilon}\bar v$. Then $w, \bar w $ are Wagner tails, where 
\[ w = \begin{cases} 
v &\text{ if } \epsilon = 1 \\
va_i^{-1} &\text{ if } \epsilon = -1 \end{cases} ~\textrm{ and }~
\bar w = \begin{cases}
\bar v^{-1} &\text{ if } \epsilon = 1 \\
\bar v^{-1}a_i &\text{ if } \epsilon = -1. \end{cases}
\]

Wagner's theorem gives a process for identifying which fixed points of a selfmap can be combined by homotopy based on the Wagner tails of the induced homomorphism. Each pair of Wagner tails corresponds to a fixed point in a particular geometric realization of the homomorphism $\phi$, and Wagner shows that these fixed points can be merged exactly when certain equalities hold among the Wagner tails. 

For our purposes, the following fact is all that we require: if fixed points $x_i$ and $x_j$ correspond to Wagner tails $w_i, \bar w_i$ and $w_j, \bar w_j$ and there is no sequence of Wagner tails $\{w_i,\bar w_i\}=\{w_0,\bar w_0\}, \{w_1,\bar w_1\}, \cdots, \{w_k,\bar w_k\}=\{w_j,\bar w_j\}$ such that $\{w_{\ell-1},\bar w_{\ell-1}\} \cap \{w_\ell,\bar w_\ell\} \neq \emptyset$ for each $\ell\in\{1,2,\cdots, k\}$,  then $x_i$ and $x_j$ cannot be combined by a homotopy.

Wagner's theorem requires that the homomorphism satisfy a ``remnant'' condition, a sort of small-cancellation property for the image words of $\phi$. This remnant condition is satisfied generically, so that $D(S \cap R_n) = D(S)$ for any set of homomorphisms, where $R_n$ is the set of endomorphisms with remnant on the free group of rank $n$.

Following \cite{bgms12}, let $V_n$ be the set of endomorphisms on the rank $n$ free group with remnant\footnote{The set denoted $V_n$ in \cite{bgms12} does not require the remnant property, but this distinction is irrelevant when discussing the density. If we let $V_n'$ be the set from \cite{bgms12} then we have $V_n = V'_n \cap R_n$ and so $D(V_n) = D(V'_n)$ since $R_n$ is generic.} having no equalities among the Wagner tails except for the repeated trivial word $w_0=\bar w_0=1$. This means that $V_n \subset W_n$, and thus $D(V_n) \le D(W_n)$. In \cite{bgms12} it is shown that $D(V_n) > 0$ and it is conjectured that $\lim_{n\to\infty} D(V_n) = 1/e$. Convincing experimental evidence is given to support the conjecture.

Our strategy is to divide the set of all endomorphisms into three disjoint subsets, one of which is $V_n$. 
For the others, let $A_{k,n} \subset G^n$ be the set of endomorphisms with remnant having some equality between a pair of Wagner tails of length $k$. Let $B_n = A_{0,n} - \bigcup_{k\ge 1}A_{k,n}$, and we have 
\begin{equation} R_n = V_n \cup B_n \cup \bigcup_{i=1}^\infty A_{i,n}.\label{unions} 
\end{equation}

The proof of our Theorem follows from two lemmas to be proved:
\begin{lem}\label{Alemma}
For any $k>0$, we have
\[ D(A_{k,n})  \le \frac{3n-2}{2n(2n-1)^k}. \]
\end{lem}

\begin{lem}\label{Blemma}
For $n\ge 2$ we have $B_n \subset W_n$.
\end{lem}

With these two lemmas, the proof of the Theorem is straightforward:
\begin{proof}[Proof of the Theorem]
Since the sets $V_n \cup B_n$ and $A_{k,n}$ are disjoint for any $k$ and, by Theorem 3.7 of \cite{wagn99}, the remnant property is generic, equation \eqref{unions} gives
\[ 1 = D(V_n \cup B_n) + D(\bigcup_{k=1}^\infty A_{k,n}). \]
The set $V_n$ is constructed so that $V_n \subset W_n$, and also we have $B_n \subset W_n$ by Lemma \ref{Blemma}. Thus $V_n \cup B_n \subset W_n$ and Lemma \ref{Alemma} gives
\begin{align*}
1 &\le D(W_n) + \sum_{k=1}^\infty D(A_{k,n}) \le D(W_n) + \sum_{k=1}^\infty \frac{3n-2}{2n(2n-1)^k} \\
&= D(W_n) + \frac{3n-2}{2n(2n-1)}\frac{1}{1-\frac1{2n-1}} = D(W_n) + \frac{3n-2}{2n(2n-2)}
\end{align*}
Thus
\begin{equation}\label{bound}
D(W_n) \ge 1 - \frac{3n-2}{2n(2n-2)},
\end{equation}
and taking the limit in $n$ gives the result, since $D(W_n) \le 1$ by definition.
\end{proof}

It remains to prove the two lemmas. Lemma \ref{Alemma} uses a purely combinatorial argument. For some fixed $n$, let $\mathcal W(k) = 2n(2n-1)^{k-1}$ be the number of words of length exactly $k$ in the free group on $n$ generators. We will sometimes refer to a subword of length $k$ as a $k$-word.

\begin{proof}[Proof of Lemma \ref{Alemma}]
Choose some endomorphism $\phi\in G_p^n$ at random by generating randomly the image words $\phi(a_i)$, and we will measure the probability that $\phi\in A_{k,n}$. Recall that $\phi \in A_{k,n}$ means that there is an equality among a pair of Wagner tails of length $k$. (More specifically homomorphisms in $A_{k,n}$ have remnant, but we may ignore this condition when computing $D(A_{k,n})$ since the remnant condition is generic.)

Let $f$ be the selfmap on a bouquet of circles whose induced homomorphism is $\phi$ obtained by the geometric realization described by Wagner in \cite{wagn99}. This map has a fixed point at the wedge point plus one fixed point for each Wagner tail pair. Call these fixed points $x_0, x_1, \dots, x_m$, with Wagner tails $w_i$ and $\bar w_i$ associated to each $x_i$. For convenience we set $w_{i+m} = \bar w_i$. For each $i\in \{1,\dots,m\}$, let $l(i) \in \{1,\dots,n\}$ be the ``location'' of $x_i$, that is, the number such that the Wagner tail $w_i$ arises from an occurence of the letter $a_{l(i)}^{\pm 1}$ inside the word $\phi(a_{l(i)})$.

Since $\phi\in A_{k,n}$ there are some distinct $i,j \in \{1,\dots,2m\}$ such that $w_i = w_j$ and $|w_i|=k$. Note that there are at most $2n$ Wagner tails of length $k$. We will assume that $i=1$ and $j=2$, and then multiply our final probability by $2n \choose 2$ to account for this choice. (Since the actual number of Wagner tails of length $k$ is typically less than $2n$, we will obtain an upper bound on the actual density.) The argument that follows will hold for any other choices of $i$ and $j$, though it must be superficially modified if either $i$ or $j$ is greater than $m$.

With these assumptions, and letting $r=l(1)$ and $s=l(2)$, the Wagner tails $w_i$ and $w_j$ will be initial $k$-words of some $\phi(a_r)$ and $\phi(a_s)$. Since $x_1 \neq x_2$ we must have $a_r \neq a_s$. Let us assume without loss of generality that $r=1$ and $s=2$. 

There are three possible ways in which we may obtain $w_1=w_2$ from a randomly chosen map. In all of the following, we assume that $p\gg 2k$ and all written products of words are assumed to be reduced.

One possible case is when $\phi(a_1) = va_1^{-1}u$ for some $v$ of length $k-1$, and $\phi(a_2) = va_1^{-1}a_2w$, so that $w_1 = w_2 = va_1^{-1}$. This restriction on $\phi(a_1)$ requires that the $k$th letter be $a_1^{-1}$, which will occur with probability $\frac{1}{2n}$, while the restriction on $\phi(a_2)$ will occur when all of the first $k+1$ letters of $\phi(a_2)$ are proscribed, which will occur with probability $\frac{1}{\mathcal W(k+1)}$. Thus the total probability for this case is $\frac{1}{2n\mathcal W(k+1)}$.

Another case is when $\phi(a_1) = va_2^{-1}a_1u$ for some $v$ of length $k-1$, and $\phi(a_2) = va_2^{-1}w$, so that $w_1 =w_2= va_2^{-1}$. This restriction on $\phi(a_1)$ requires that the $(k-1)$st and $k$th letters be specified, and thus has probability $\frac1{2n(2n-1)}$ (once the first letter is specified, there are only $2n-1$ remaining choices for the second), while the restriction on $\phi(a_2)$ requires that the initial $k$ letters all be specified, which has probability $\frac{1}{\mathcal W(k)}$. Thus the total probability for this case is 
\[ \frac{1}{2n(2n-1)\mathcal W(k)} = \frac{1}{2n\mathcal W(k+1)}. \]

The final case is when $\phi(a_1) = va_1u$ where $v$ has length $k$ and does not end with $a_2^{-1}$, and $\phi(a_2) = va_2w$, so that $w_1=w_2=v$. This restriction on $\phi(a_1)$ requires that the $k$th letter be $a_1$ and the $(k-1)$st letter be anything but $a_2^{-1}$ or $a_1^{-1}$, which will occur with probability $\frac{1}{2n}(1-\frac{1}{2n-1})$, while the restriction on $\phi(a_2)$ occurs with probability $\frac{1}{\mathcal W(k+1)}$. Thus the total probability for this case is 
\[ \frac{1}{2n}\left(1-\frac{1}{2n-1}\right)\frac{1}{\mathcal W(k+1)} = \frac{n-1}{n(2n-1)} \frac{1}{\mathcal W(k+1)}. \]

The total probability that $\phi \in A_{k,n}$ will equal $\frac{|A_{k,n} \cap G^n_p|}{|G^n_p|}$ when $p$ is sufficiently greater than $k$. Summing the probabilities of the three cases above, along with the factor of $2n \choose 2$, gives
\begin{align*} 
\frac{|A_{k,n} \cap G^n_p|}{|G^n_p|} 
&\le {2n\choose 2} \frac{1}{\mathcal W(k+1)} \left( \frac{1}{2n} + \frac{1}{2n} + \frac{n-1}{n(2n-1)} \right) \\
&= \frac{2n(2n-1)}{2} \frac{1}{\mathcal W(k+1)} \frac{3n-2}{n(2n-1)} 
= \frac{3n-2}{2n(2n-1)^{k}}.
\end{align*}
Taking the lim inf in $p$ gives the result.
\end{proof}

Next we prove Lemma \ref{Blemma}, that all homomorphisms in $B_n$ are Wecken.
This is a topological argument which requires a specific geometric construction on the bouquet of circles. To facilitate this, we will use the following notation:

Let $X$ be the bouquet of $n$ circles whose fundamental group is $\pi_1(X) = \langle a_1, \dots, a_n\rangle$. Let $|a_i|\subset X$ be the circle in $X$ corresponding to the oriented loop $a_i\in \pi_1(X)$. We will parameterize each loop $|a_i|$ as an interval $[0,1]$ with endpoints identified. For a real number $x \in [0,1]$, let $(x)_{a_i} \in |a_i|$ be the corresponding point under this parameterization, so that $(0)_{a_i} = (1)_{a_i} = x_0$, the base point of $X$. For real numbers $x,y \in [0,1]$, let $[x,y]_{a_i}\subset |a_i|$ denote the interval-like segment from $(x)_{a_i}$ to $(y)_{a_i}$. When convenient we will use negative parameters, e.g.\ $[-\epsilon,0]_{a_i} = [1-\epsilon,1]_{a_i}$.

Every homomorphism $\phi:\langle a_1, \dots a_n\rangle$ has a \emph{linear realization} $f:X \to X$ defined as follows: $f$ maps $x_0$ to itself, and maps certain intervals linearly onto circles $|a_i|$. In particular, if the reduced form of $\phi(a_j)$ is $\phi(a_j) = a_{i_1}^{\eta_1} \dots a_{i_k}^{\eta_k}$ for $\eta_l \in \{+1, -1\}$, then $f$ maps the interval $[\frac{l-1}{k}, \frac{l}{k}]_{a_j}$ linearly onto $[0,1]_{a_{i_l}} = |a_{i_l}|$, either preserving or reversing the orientation according to the sign of $\eta_l$. If $k=0$ then $f(|a_j|) = x_0$. 

This linear realization is closely related to Wagner's geometric realization, which additionally makes $f$ constant on a neighborhood of the base point. Wagner's realization is homotopic to the linear realization described above.

We are now ready to prove Lemma \ref{Blemma}:
\begin{proof}[Proof of Lemma \ref{Blemma}]
We will show that $\phi$ is Wecken by explicitly constructing a Wecken map on a bouquet of $n$ circles whose induced homomorphism is $\phi$. 

First we consider Wagner's realization of $\phi$, which we denote by $f$. Since $\phi$ has remnant, by Wagner's theorem the fixed point classes of $f$ are exactly determined by the equalities among the Wagner tails of $\phi$. By our assumption in the theorem, these equalities exist only for trivial Wagner tails, and so all fixed point classes of $f$ are singleton essential classes except perhaps for the class containing the base point $x_0$. 

Let $\mathbb F$ be the fixed point class of $f$ containing $x_0$, say $\mathbb F = \{ x_0, x_{1}, \dots, x_{k}\}$, where each $x_i$ for $i>0$ is a fixed point of index $-1$ and $x_0$ has index $1$. When we change $f$ by homotopy to the linear realization $f'$, this fixed point class will become $\mathbb F' = \{x_0\}$. Since the index of a fixed point class is homotopy invariant, we will have $\ind \mathbb F' = \ind \mathbb F = 1-k$. 

If $k > 1$ then $\mathbb F'$ is an essential fixed point class consisting of one point, and so $f'$ (and thus $\phi$) is Wecken as desired. 
It remains to consider the case where $k=1$, in which case we must show that the fixed point of $f'$ at $x_0$ can be removed by a homotopy. Without loss of generality let us assume that $x_1 \in |a_1|$, which means that the word $\phi(a_1)$ begins with $a_1$ and no other words $\phi(a_i)$ begin or end with $a_i$. The case in which $\phi(a_1)$ ends (rather than begins) with $a_1$ uses an analogous argument.

\begin{figure}
\def\svgwidth{6in}
\begin{center}
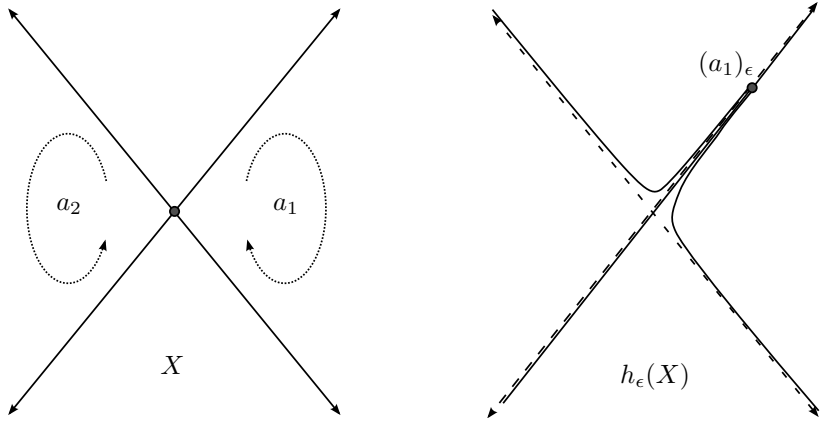
\end{center}
\caption{Action of $h_\epsilon$ in a neighborhood of the base point for $n=2$\label{pullpic}}
\end{figure}

In order to remove the fixed point, we choose some small $\epsilon$ and we will construct a map $h_\epsilon: X\to X$ which is homotopic to the identity. Let $U_\epsilon$ be the $\epsilon$-neighborhood of $x_0$ with respect to the linear parameterization of $X$ given above. We define $h_\epsilon$ as follows: $h_\epsilon$ is the identity map outside of $U_\epsilon$. Informally, the effect of $h$ is to ``pull'' the basepoint $x_0$ to the point $(\epsilon/2)_{a_1}$ as in Figure \ref{pullpic}.
Specifically, inside $U_\epsilon$, the map $h_\epsilon$ carries intervals linearly as follows:
\begin{align*}
h_{\epsilon}([0,\epsilon]_{a_1}) &= [\epsilon/2, \epsilon]_{a_1}, \\
h_{\epsilon}([0,\epsilon/2]_{a_i}) &= [\epsilon/2,0]_{a_1} \quad \text{for }i \neq 1, \\
h_{\epsilon}([\epsilon/2, \epsilon]_{a_i}) &= [0, \epsilon]_{a_i} \quad \text{for }i \neq 1, \\
h_{\epsilon}([-\epsilon, -\epsilon/2]_{a_i}) &= [-\epsilon, 0]_{a_i} \quad \text{for all }i, \\
h_{\epsilon}([-\epsilon/2, 0]_{a_i}) &= [0,\epsilon/2]_{a_1} \quad \text{for all }i.
\end{align*} 
Image intervals written $[x,y]_*$ with $x<y$ indicate that the map is constructed to be orientation preserving, while $y<x$ indicates that the map is orientation reversing.

Varying the value of $\epsilon$ shows that $h_0$ is the identity and thus that $h_\epsilon$ is homotopic to the identity on $X$. Now let $\bar f = f' \circ h_\epsilon$, so $\bar f$ is homotopic to $f$ and these maps agree outside of $U_\epsilon$. We now claim that $\bar f$ has no fixed points on $U_\epsilon$, and thus that the fixed point of $f'$ at $x_0$ has been removed.

It suffices to check that $\bar f$ is fixed-point free on various intervals in $U_\epsilon$. For example observe that 
\[ \bar f([0,\epsilon]_{a_1}) = f'([\epsilon/2, \epsilon]_{a_1}) = [\delta/2, \delta]_{a_1}, \]
for some $\delta > \epsilon$. (The map $f'$ is linear with ``slope'' greater than 1.) Since $\bar f$ maps $[0,\epsilon]_{a_1}$ linearly to $[\delta/2, \delta]_{a_1}$ with $\delta > \epsilon$, there are no fixed points of $\bar f$ in $[0,\epsilon]_{a_1}$. A similar argument shows that $\bar f$ has no fixed points on $[\epsilon/2, \epsilon]_{a_i}$ for $i\neq 1$ or $[-\epsilon, -\epsilon/2]_{a_i}]$ for any $i$.

Next we consider $[0,\epsilon/2]_{a_i}$ for $i\neq 1$. We have $\bar f([0,\epsilon/2]_{a_i}) = f'([\epsilon/2,0]_{a_1})$ which is an initial segment of $|a_1|$ with orientation reversed, and since $i\neq 1$ there are no fixed points of $\bar f$ here (the base point is not fixed since the orientation of the intervals is reversed). Similar arguments show that $\bar f$ has no fixed points on $[-\epsilon/2, 0]_{a_i}$ for any $i$.
\end{proof}

We note in conclusion that the arguments in our main theorem can be used to prove the conjecture of \cite{bgms12} that $\lim_{n\to\infty} D(V_n) = 1/e$. Using counting arguments similar to those in \cite{bgms12}, it is not difficult to show that $D(A_{0,n}) = 1- 1/e$. (A homomorphism $\phi$ is in $A_{0,n}$ if and only if at least one word $\phi(a_i)$ begins or ends with $a_i$. The required counting argument is essentially contained in \cite[Theorem 11]{bgms12}.) The proof of our main theorem established that $\lim_{n\to\infty} D(V_n \cup B_n) = 1$, and thus we have
\[ 1 \le \lim_{n\to\infty} D(V_n) + D(B_n) \le \lim_{n\to\infty} D(V_n) + D(A_{0,n}) = \lim_{n\to\infty} D(V_n) + 1 - 1/e, \]
and thus $\lim_{n\to\infty} D(V_n) \ge 1/e$. But Theorem 11 of \cite{bgms12} shows that $\lim_{n\to\infty} D(V_n) \le 1/e$ and the conjecture is proved. This argument also demonstrates that $\lim_{n\to\infty} D(B_n) = 1 - 1/e$.

We also note that the proven convergence of $D(W_n)$ to 1 by \eqref{bound} is somewhat slow. The values on the right side of \eqref{bound} do not exceed 0.9 until $n=8$, and do not exceed $0.99$ until $n=76$. It seems plausible that the convergence is actually faster than this, perhaps exponentially fast. In fact it remains an open question whether $D(w_n) = 1$ for all (or any) $n$. 

We have made no effort to compute the strict density $\rho(W_n)$, which would require showing that the limit (rather than lim inf) of the ratios $\frac{|W_{n} \cap G^n_p|}{|G^n_p|}$ exists. This does not seem accessible to our methods, since we work always with subsets of $W_n$ and never obtain upper bounds. In an effort to clear up some details unaddressed in \cite{bgms12}, we have been able to show that $\rho(V_n)$ exists for each $n$, and we prove this in the appendix.

\section*{Appendix: Existence of $\rho(V_n)$}
In this brief appendix we fill in a detail which was untreated in \cite{bgms12}, namely the existence of the strict density $\rho(V_n)$ for each $n\ge 2$. Let $V'_n$ be the set of endomorphisms with no equalities among the Wagner tails except for $w_0=\bar w_0 = 1$. The set $V_n$ additionally requires the remnant property, so that $V_n = V'_n \cap R_n$. Since $\rho(R_n)=1$, if $\rho(V'_n)$ exists then $\rho(V_n)$ will exist and we will have $\rho(V_n) = \rho(V'_n)$.

It suffices then to show that $\rho(V'_n)$ exists. (In fact the set called $V_n$ in \cite{bgms12} is exactly $V'_n$.) Let $\beta_p=\frac{|V'_n \cap G_p^n|}{|G_p^n|}$. We will show that the
sequence $\langle \beta_p\rangle$ is monotone decreasing, which is vaguely suggested by the experimental data presented in \cite{bgms12}. Thus
\[ \rho(V'_n) = \lim_{p \to \infty} \beta_p \]
exists.

We make use of general counting formulas used throughout \cite{bgms12}, most importantly 
\[ |G_p| = \sum_{k=1}^p \mathcal W(k) = \frac{n(2n-1)^p -1}{n-1} \]

Let $\mathcal S(k,p)=|G_p - G_{k-1}|$ be the number of
words of length between $k$ and $p$. Then \[\mathcal
S(k,p)=\frac{n(2n-1)^p -1}{n-1} - \frac{n(2n-1)^{k-1}
-1}{n-1}=\frac{n(2n-1)^{k-1}((2n-1)^{p-k+1}-1)}{n-1}.\]
Let $\mathcal U(k,p) =\frac{\mathcal S(k,p)}{|G_p|}$ be the probability of choosing a word with length between $k$ and $p$ when we choose a word in $G_p$ at random. Then we have
\begin{align*}
\mathcal U(k,p) &=\frac{\mathcal
S(k,p)}{|G_p|}=\frac{n(2n-1)^{k-1}((2n-1)^{p-k+1}-1)}{n-1}\cdot\frac{n-1}{n(2n-1)^p
-1} \\
&=\frac{n(2n-1)^{k-1}((2n-1)^{p-k+1}-1)}{n(2n-1)^p -1}
\end{align*}

Now we are ready for our result.

\begin{thm}
The sequence $\langle \beta_p\rangle$ is monotone decreasing.
\end{thm}

\newcommand\inv{^{-1}}

\begin{proof}
Let $V_n^{\prime c}$ be the complement of $V'_n$, and for each $p$ we will show that $\frac{|V_n^{\prime c} \cap
G_p^n|}{|G_p^n|}\leq \frac{|V_n^{\prime c} \cap G_{p+1}^n|}{|G_{p+1}^n|}$.
Then we have $\frac{|V'_n \cap G_p^n|}{|G_p^n|}\geq \frac{|V'_n \cap
G_{p+1}^n|}{|G_{p+1}^n|}$ and thus $\langle \beta_p\rangle$ is monotone
decreasing.

We only check the subset of $V'^c_n$ consisting of maps $\phi$ with $\phi(a_1)=va_1\inv u$ for some word $v$
of length $k-1$ and $\phi(a_2)=va_1\inv a_2w$. Similar arguments will suffice to show that the proportion of other maps in $V_n^c$ is increasing. Let the proportion of such maps in $G_p^n$ be $\gamma_p$, and we will show that $\langle \gamma_p \rangle$ is increasing.

In $G_p^n$, since $|\phi(a_1)|\geq k$, word $\phi(a_1)$ will occur
with probability $\mathcal U(k,p)$. The restriction on $\phi(a_1)$
requires that the $k$th letter be $a_1\inv$, which will occur with
probability $\frac{1}{2n}$. Meanwhile, since $|\phi(a_2)|\geq k+1$,
word $\phi(a_2)$ will occur in $G_p^n$ with probability $\mathcal
U(k+1,p)$. The restriction on $\phi(a_2)$ requires that the initial
$k+1$ letters be specified, which will occur with probability
$\frac{1}{\mathcal W(k+1)}$. Thus the total probability for this
case in $G_p^n$ is
$$
\gamma_p = \frac{\mathcal U(k,p)}{2n}\cdot\frac{\mathcal U(k+1,p)}{\mathcal
W(k+1)}.
$$
Similarly, for the same fixed $k$, the probability of occurrence of
this case in $G_{p+1}^n$ is
$$
\gamma_{p+1} = \frac{\mathcal U(k,p+1)}{2n}\cdot\frac{\mathcal U(k+1,p+1)}{\mathcal
W(k+1)}.
$$

Now we have
\begin{align*}
\gamma_p &= \frac{\mathcal U(k,p)}{2n} \cdot \frac{\mathcal U(k+1,p)}{\mathcal W(k+1)} \\
%= \frac{\mathcal U(k,p)}{\mathcal U(k,p+1)} \cdot \frac{\mathcal U(k,p+1)}{2n} \cdot \frac{\mathcal U(k+1,p)}{\mathcal W(k+1)} 
&= \frac{\mathcal U(k,p)}{\mathcal U(k,p+1)} \cdot \frac{\mathcal U(k+1,p)}{\mathcal U(k+1,p+1)} \cdot \frac{\mathcal U(k,p+1)}{2n} \cdot \frac{\mathcal U(k+1,p+1)}{\mathcal W(k+1)}  \\
&= \frac{\mathcal U(k,p)}{\mathcal U(k,p+1)} \cdot \frac{\mathcal U(k+1,p)}{\mathcal U(k+1,p+1)} \cdot \gamma_{p+1}
\end{align*}

%Since $\frac{\mathcal U(k+1,p)}{\mathcal
%W(k+1)}\approx\frac{\mathcal U(k+1,p+1)}{\mathcal W(k+1)}\approx
%\frac{1}{2n(2n-1)^{k-1}}$, if we have $\mathcal U(k,p)\leq\mathcal
%U(k,p+1)$ for any $k$ with $k\leq p+1$ ($k$ is at most $p$ in
%$G_p^n$), then we have $\frac{|V_n^c \cap G_p^n|}{|G_p^n|}\leq
%\frac{|V_n^c \cap G_{p+1}^n|}{|G_{p+1}^n|}$.

We will show that $\frac{\mathcal U(k,p)}{\mathcal U(k,p+1)}\leq 1$
for each $k$ with $k\leq p$. In this case the first two factors above are at most 1, and we will have shown that $\gamma_p \le \gamma_{p+1}$. We have
\begin{align*}
\frac{\mathcal U(k,p)}{\mathcal U(k,p+1)}
&=\frac{n(2n-1)^{p+1} -1}{n(2n-1)^{k-1}((2n-1)^{p-k+2}-1)}\cdot
\frac{n(2n-1)^{k-1}((2n-1)^{p-k+1}-1)}{n(2n-1)^p -1}\\
&=\frac{(n(2n-1)^{p+1}
-1)((2n-1)^{p-k+1}-1)}{((2n-1)^{p-k+2}-1)(n(2n-1)^p -1)}\\
&=\frac{n(2n-1)^{2p-k+2}-n(2n-1)^{p+1}-(2n-1)^{p-k+1}+1}
{n(2n-1)^{2p-k+2}-(2n-1)^{p-k+2}-n(2n-1)^p+1}\\
&=\frac{n(2n-1)^{2p-k+2}+1-(2n-1)^{p-k+1}(n(2n-1)^k
+1)}{n(2n-1)^{2p-k+2}+1-(2n-1)^{p-k+1}((2n-1)+n(2n-1)^{k-1})}\\
&\leq 1, \end{align*}
where the inequality holds because  $(2n-1)+n(2n-1)^{k-1}\leq n(2n-1)^k
+1$.

\end{proof}


\begin{thebibliography}{1}

\bibitem{bgms12}
J.~Brimley, M.~Griisser, A.~Miller, and P.~C. Staecker, \emph{The {W}ecken
  property for random maps on surfaces with boundary}, Topology and Its
  Applications \textbf{159} (2012) 3662--3676. Arxiv eprint 1109.0218.

\bibitem{jian84}
B.~Jiang, \emph{Fixed points and braids}, Inventiones Mathematicae \textbf{75}
  (1984) 69--74.

\bibitem{wagn99}
J.~Wagner, \emph{An algorithm for calculating the {N}ielsen number on surfaces
  with boundary}, Transactions of the American Mathematical Society
  \textbf{351} (1999) 41--62.

\bibitem{weck41}
F.~Wecken, \emph{Fixpunktklassen {I}, {II}, {III}}, Mathematische Annalen
  \textbf{117, 118} (1941, 1942) 659--671, 216--234, 544--577.

\end{thebibliography}
\end{document}